\documentclass{article}

\usepackage{amssymb,amsmath,amsthm,fullpage}
\usepackage{tikz-cd}
\usepackage{enumitem}
\usepackage{mathtools}
\usepackage[titletoc,title]{appendix}

\theoremstyle{plain}
\newtheorem*{theorem*}{Theorem}
\newtheorem{theorem}{Theorem}[section]
\newtheorem{lemma}[theorem]{Lemma}

\newtheorem*{conjecture*}{Conjecture}

\theoremstyle{definition}
\newtheorem{remark}[theorem]{Remark}
\newtheorem{example}[theorem]{Example}

\DeclareMathOperator{\deff}{def}

\DeclareMathOperator{\Proj}{Proj}
\DeclareMathOperator{\Sym}{Sym}
\DeclareMathOperator{\Sing}{Sing}
\DeclareMathOperator{\id}{id}
\DeclareMathOperator{\GL}{GL}
\newcommand{\ZZ}{\mathbb{Z}}
\newcommand{\QQ}{\mathbb{Q}}
\newcommand{\RR}{\mathbb{R}}
\newcommand{\CC}{\mathbb{C}}
\newcommand{\PP}{\mathbb{P}}

\newcommand{\TT}{\mathbb{T}}
\newcommand{\EEs}{\mathcal{E}}

\newcommand{\NX}{N_{X/\mathbb{P}^{N}}}
\newcommand{\NXt}{N_{X/\mathbb{P}^{N}}(-1)}
\newcommand{\HHs}{\mathcal{H}}

\newcommand{\OOs}{\mathcal{O}}

\begin{document}

\date{}
\title{Linear recurrence sequences and the duality defect conjecture}
\author{Grayson Jorgenson}
\maketitle

\begin{abstract}
It is conjectured that the dual variety of every smooth nonlinear subvariety of dimension $> \frac{2N}{3}$ in projective $N$-space is a hypersurface, an expectation known as the duality defect conjecture. This would follow from the truth of Hartshorne's complete intersection conjecture but nevertheless remains open for the case of subvarieties of codimension $> 2$. A combinatorial approach to proving the conjecture in the codimension $2$ case was developed by Holme, and following this approach Oaland devised an algorithm for proving the conjecture in the codimension $3$ case for particular $N$. This combinatorial approach gives a potential method of proving the duality defect conjecture in many of the cases by studying the positivity of certain homogeneous integer linear recurrence sequences. We give a generalization of the algorithm of Oaland to the higher codimension cases, obtaining with this bounds the degrees of counterexamples would have to satisfy, and using the relationship with recurrence sequences we prove that the conjecture holds in the codimension $3$ case when $N$ is odd.
\end{abstract}

\section{Introduction}

This article concerns the \emph{duality defect conjecture}, which may have first been posed by Alan Landman in 1974:

\begin{conjecture*}
Let $X$ be a smooth nonlinear subvariety of $\PP^N$ with $\dim(X) > \frac{2N}{3}$. Then $X^\vee$ is a hypersurface, or equivalently $\deff(X) = 0$.
\end{conjecture*}

Here $X^\vee$ denotes the \emph{dual variety} of $X$, the subvariety of dual projective space $\PP^{N\vee}$ consisting of all hyperplanes tangent to $X$, and $\deff(X) = N - 1 - \dim(X^\vee)$ is called the \emph{duality defect} of $X$, which provides a measure of how far $X^\vee$ is from being a hypersurface. Our aim is to give a generalization of an algorithm of Oaland \cite{oaland1} for proving the conjecture in the codimension $3$ case to the higher codimension cases of the conjecture and to prove the following two results:

\setcounter{section}{4}
\setcounter{theorem}{7}
\begin{theorem}
If $X$ is a smooth nonlinear codimension $m$ subvariety of $\PP^N$ with duality defect $r > 0$ and $N - m \geq \frac{3N-2}{4}$, then $$\deg(X)\leq \sum_{j = 0}^{m}\left(\frac{N - m - r}{2}\right)^{j}.$$
\end{theorem}

\setcounter{section}{5}
\setcounter{theorem}{0}
\begin{theorem}
The duality defect conjecture in the codimension $3$ case is true for $\PP^N$ when $N$ is odd.
\end{theorem}

Some of the motivation for this work comes from a well-known conjecture of Hartshorne \cite{hartshorne3} on complete intersections.

\begin{conjecture*}
If $X$ is a smooth subvariety of $\PP^N$ with $ \dim(X) >\frac{2N}{3}$, then $X$ is a complete intersection.
\end{conjecture*}

This conjecture still remains open today, over four decades after its declaration, even for the case of subvarieties of codimension $2$. Its truth would imply that of the duality defect conjecture, see for instance \cite[Proposition 3.1]{ein1} or \cite[Proposition 2.1]{holme1}. So far the duality defect conjecture has proven to be a more tractable problem; it is known to be true in the codimension $2$ case, with proofs given by Ein \cite[Theorem 3.4]{ein1} and independently by Holme and Schneider \cite[Corollary 6.4]{holme2}. A proof of the duality defect conjecture would provide evidence for Hartshorne's conjecture, and any counterexample would also serve to disprove Hartshorne's conjecture.

We are interested in a combinatorial approach to proving the duality defect conjecture that Holme developed which was used to give an alternate proof of the codimension $2$ case \cite{holme1}. This approach is promising as it may be extendable to the higher codimension cases. Oaland \cite{oaland1} used Holme's approach and results of Ein \cite[Theorem 2.3]{ein1} and Zak \cite[Corollary 7.4]{fulton2} to create an algorithm for proving the conjecture in the codimension $3$ case for individual values of $N$, and was able to use this to prove the conjecture in that case for $N = 10,11,\ldots, 141$.

Specifically, given a smooth subvariety $i:X\hookrightarrow \PP^N$ of dimension $n$ the combinatorial approach works by considering the class of the \emph{conormal variety} $Z(X) = P(\NXt ^\vee)\hookrightarrow \PP^N\times\PP^{N\vee}$ in the Chow ring $A(\PP^N\times\PP^{N\vee})$. This class may be written as $$[Z(X)] = \delta_0(X)s^N t + \ldots + \delta_n(X) s^{N-n}t^{n + 1},$$ for some $\delta_j(X)\in\ZZ$, where if $pr_1:\PP^N\times\PP^{N\vee}\rightarrow \PP^N$ and $pr_2: \PP^N\times\PP^{N\vee}\rightarrow \PP^{N\vee}$ are the projection maps, $s = pr_1^{*}([H])$ and $t = pr_2^{*}([H^\prime])$ for hyperplanes $H\subseteq \PP^N$, $H^\prime\subseteq \PP^{N\vee}$. It is shown by Holme \cite{holme1} that for each $j$, we have $\delta_j(X)\geq 0$, and, if $\delta_j(X) = 0$ for $j = 0,\ldots,r - 1$, and $\delta_{r}(X) \neq 0$ for some $r\in\{0,\ldots,n\}$, then in fact $\deff(X) = r$. Thus the duality defect of $X$ can be computed from the coefficients of the class of $Z(X)$. In particular, the dual variety of $X$ is a hypersurface if and only if $\delta_0(X) \neq 0$.

Additionally, each $\delta_j(X) = \deg(s_{n - j}(\NXt^\vee))$, where $s_{j}(\NXt^\vee)$ denotes the $j$th Segre class of the bundle $\NXt^\vee$ using the conventions of Fulton \cite{fulton1}. By applying properties of singular cohomology for complex projective varieties one may show that when $n\geq \frac{3N - 2}{4}$, for our purposes we may assume $c(\NXt) = 1 + c_1 h + \ldots + c_{N - n} h^{N - n}$ for some $c_j\in \ZZ$, where $c(\NXt)$ is the total Chern class of $\NXt$ and $h = i^*([H])$. This allows us to write $$s(\NXt^\vee) = \frac{1}{c(\NXt^\vee)} = 1 + s_1 h + \ldots + s_n h^n$$ for some $s_j\in\ZZ$, and where each $\delta_j(X) = \deg(X) s_{n-j}$. These $s_j$ form part of a \emph{homogeneous integer linear recurrence sequence} of order $N - n$:
\begin{equation}
\begin{aligned}
s_0 &= 1,\\
s_j &= \sum_{q = 1}^{j} (-1)^{q+1} c_{q}s_{j-q}, \text{ for } j = 1,\ldots,N- n - 1\\
s_j &= \sum_{q = 1}^{N - n} (-1)^{q + 1}c_{q}s_{j - q}, \text{ for } j \geq N - n.
\end{aligned}
\end{equation}

The algorithm of Oaland \cite[Kapittel 5]{oaland1} works in the codimension $3$ case by showing that if $X$ has positive defect then there are only finitely many possibilities for $c_1, c_2, c_3$, and the possible sequences of $s_j$ can be computed. One of the ideas of Holme \cite{holme1} applied in Oaland's work is to use positivity properties that the \emph{Schur polynomials} in the Chern classes of $\NXt$ satisfy to obtain inequalities involving the $c_j$. Using this idea, we find inequalities that extend the algorithm of Oaland to the higher codimension cases of the conjecture. These inequalities also follow from a recent result of Huh \cite[Theorem 21]{huh1} which specifies additional constraints the coefficients of the class of a subvariety in $A(\PP^N\times\PP^{N\vee})$ must satisfy, and are key to our proof of Theorem 4.8. By running the generalized algorithm we show that the duality defect conjecture is true in the codimension $3$ case for $N = 10,11,\ldots,201$, in the codimension $4$ case for $N = 14,15, \ldots,50$, and in the codimension $5$ case for $N = 18,19\ldots,23$.

Alternatively, if one can show there are simply no recurrence sequences of this form with $s_n = \ldots = s_{n - r + 1} = 0$ and $s_j > 0$ for $j \leq n - r$ for appropriate $r > 0$, then there can be no positive defect smooth nonlinear subvarieties of $\PP^N$ of dimension $n$. In this way, a better understanding of the positivity of homogeneous integer linear recurrence sequences can have geometric implications. This is the idea of our proof of Theorem 5.1, where we reduce the necessary work to proving the following purely number-theoretic statement.

\setcounter{theorem}{1}
\begin{lemma}
Let $(u_j)_{j\in\mathbb{Z}_{\geq0}}$ be the homogeneous integer linear recurrence sequence defined by the initial conditions $u_0 = 0, u_1 = 0, u_2 = 1$, and $u_j = c_1 u_{j - 1} - c_2 u_{j - 2} + c_3 u_{j - 3}$ for $j>2$, where $c_1,c_2,c_3\in\mathbb{Z}_{>0}$. Suppose there exists an $m\in\ZZ_{>2}$ with $u_m = u_{m+1} = 0$ and $u_{j} > 0$ for $j = 2,\ldots,m-1$. Then $m = 4$ or $6$.
\end{lemma}

This article is organized as follows. In Section 2 we introduce conventions, notations, and facts that will be used throughout. In Section 3 we define the needed concepts from projective duality and describe how the duality defect of a smooth projective subvariety may be computed from the coefficients of the class of its conormal variety. In Section 4 we highlight the connection between the duality defect conjecture and homogeneous integer linear recurrence sequences and then give a generalization of the algorithm of Oaland to the higher codimension cases of the conjecture. Here we also prove Theorem 4.8. Lastly, Section 5 is dedicated to the proof of Theorem 5.1.

\textit{Acknowledgements}
I would like to thank Paolo Aluffi for all his guidance and support, and for introducing me to the duality defect problem. I also wish to thank Mark van Hoeij for pointing out to me that Lemma 5.2 could be sharpened to its current form. This work was partially supported by NSA grant H98230-16-1-0016.

\setcounter{section}{1}
\section{Preliminaries}

Throughout we take our base field to be $\CC$. For an algebraic vector bundle $E$ of rank $e$ on a variety $X$ of dimension $n$, we define its projective bundle to be the $X$-scheme $P(E) = \Proj(\Sym_{\OOs_X}(\EEs^{\vee}))$ where $\EEs$ is the locally free sheaf of sections of $E$, following the conventions in Fulton \cite{fulton1}. Our definition is contrary to the convention used by Holme \cite{holme1} and Grothendieck, where our $P(E)$ would be denoted $\PP(\EEs^\vee)$. Here $\EEs^{\vee}$ denotes the dual of $\EEs$, that is, the sheaf $\EEs^{\vee} = \HHs om_{\OOs_X}(\EEs,\OOs_X)$. We will also denote the dual bundle of $E$ by $E^\vee$ and by $E(d)$ the bundle $E\otimes O_X(d)$.

Additionally, we use Fulton's conventions for Chern and Segre classes \cite[Chapter 3]{fulton1}. To simplify notation, we will denote the $j$th Segre class of $E$ by $s_j(E) = s_j(E)\cap [X]$ which lives in the Chow group $A_{n-j}(X)$, and by $c_j(E) = c_j(E)\cap [X] \in A_{n-j}(X)$ the $j$th Chern class of $E$. The total Chern class of $E$ will be denoted $c(E) = 1 + c_1(E) + \ldots + c_e(E)$, and the Chern and Segre polynomials in indeterminant $t$ of $E$ will be denoted by $c_t(E) = \sum_{j=0}^{\infty} c_j(E)t^j$, $s_t(E) = \sum_{j=0}^{\infty}s_j(E)t^j$, respectively. The Chern and Segre classes of $E$ are related via $$s_t(E) = \frac{1}{c_t(E)}$$ and so:

\begin{lemma}
We have $s_0(E) = 1$, and $s_j(E) = - \sum_{q = 1}^{j}c_{q}(E)s_{j-q}(E)$ for every $j \in \ZZ_{>0}$.
\end{lemma}

Suppose $\lambda = (\lambda_1,\ldots,\lambda_m)$ represents a partition of an integer $m\in\ZZ_{>0}$ into integers $e \geq \lambda_1 \geq \ldots \geq \lambda_m \geq 0$. We define the \emph{Schur polynomial} corresponding to $\lambda$ in the Chern classes of $E$ to be $$\Delta_\lambda (E) = \det(c_{\lambda_i + j - i}(E))_{1\leq i,j\leq m} = \det \begin{bmatrix}
c_{\lambda_1}(E) & c_{\lambda_1 + 1}(E) & \cdots & c_{\lambda_1 + m - 1}(E)\\
c_{\lambda_2 - 1}(E) & c_{\lambda_2}(E) & \cdots & c_{\lambda_2 + m - 2}(E)\\
\cdots & \cdots & \cdots & \cdots\\
c_{\lambda_m - m + 1}(E) & c_{\lambda_m - m + 2}(E) & \cdots & c_{\lambda_m}(E)
\end{bmatrix},$$ adopting the notation $c_j(E) = 0$ for $j < 0$. From Fulton \cite[Example 12.1.7]{fulton1} we have:

\begin{theorem}
If $E$ is globally generated, then $\Delta_\lambda(E)\in A^{\geq}_{n-m}(X)$ where $A^{\geq}_{n-m}(X)$ denotes the subset of $A_{n-m}(X)$ consisting of the classes that can be represented by nonnegative cycles.
\end{theorem}

One of the key ingredients to our approach is a topological result due to Larsen \cite{larsen1} concerning the singular cohomology of complex projective varieties.

\begin{theorem}
Let $i: X\hookrightarrow \PP^N$ be a smooth subvariety of codimension $r$. Then the maps $H^j(\PP^N,\ZZ)\rightarrow H^j(X,\ZZ)$ induced by $i$ are isomorphisms for $j \leq N - 2r$.
\end{theorem}

We abuse notation, using the same symbols to denote both varieties in the algebraic context and their analytifications. Connecting the Chow and singular theories, there are cycle maps $A^j(X)\rightarrow H_{2n-2j}(X)$ inducing a group homomorphism ${cl}_* : A(X)\rightarrow H_*(X)$ covariant for proper morphisms, and after composition with the isomorphisms from Poincar\'e duality, a ring homomorphism ${cl}^* : A(X)\rightarrow H^*(X,\ZZ)$ contravariant for morphisms of smooth varieties. Since $\PP^N$ admits a cellular decomposition $\emptyset \subseteq \PP^1 \subseteq \PP^2 \subseteq \ldots \subseteq \PP^N$, the cycle map ${cl}_* : A(\PP^N)\rightarrow H_{*}(\PP^N)$ is an isomorphism. A reference for these facts is Fulton \cite[Chapter 19]{fulton1}.

If $\alpha\in A^j(X)$, we define the \emph{degree} of $\alpha$ to be $\deg(\alpha) = a$, with $a\in \ZZ$ such that $i_*(\alpha) = a[H]^{N - n + j}\in A(\PP^N)$, where $[H]$ is the class of a hyperplane in $\PP^N$. Similarly, for an element $\alpha\in H_{j}(X)$, we define $\deg(\alpha)\in \ZZ$ using the isomorphism $H_*(X)\cong \ZZ[s]/(s^{N+1})$. Since the cycle map ${cl}_*$ is covariant for proper morphisms we have that if $\alpha\in A_j(X)$ for some $j$, then $\deg(\alpha) = \deg({cl}_*(\alpha))$.

\section{Dual varieties and duality defect}

We denote \emph{dual projective space} of dimension $N$ by $\PP^{N\vee}$. This is projective space of dimension $N$, but with its $\CC$-points identified with the hyperplanes of $\PP^N$. In the classical language: $$\PP^{N\vee} = \{H\subseteq \PP^N\mid H \text{ is a hyperplane}\}.$$

If we choose coordinates $\PP^N (x_0:\ldots:x_N)$, then the natural identification of $\PP^N$ with $\PP^{N\vee}$ is defined so that $(a_0:\ldots: a_N)$ corresponds to the hyperplane of $\PP^N$ cut out by $a_0 x_0 + \ldots + a_N x_N \in \CC[x_0,\ldots,x_N]$. Suppose $i: X\hookrightarrow\PP^N$ is a subvariety of dimension $n$, where $i$ denotes the inclusion map, and let $(F_1,\ldots, F_r)$ be its homogeneous ideal. For $p\in X$, we define the \emph{embedded tangent space} of $X$ at $p$, denoted $\TT_{X,p}$, to be the linear subvariety of $\PP^N$ cut out by the polynomials $$\frac{\partial F_j}{\partial x_0}(p) x_0 + \ldots + \frac{\partial F_j}{\partial x_N}(p) x_N.$$ We define the \emph{conormal variety} of $X$ to be $$Z(X) = \overline{\{(p,H)\mid \TT_{X,p}\subseteq H, p\in X^\circ\}}\subseteq \PP^N\times\PP^{N\vee},$$ where $X^\circ$ denotes the smooth locus of $X$. From now on we assume $X$ is smooth, in which case taking the closure of the set in the definition is unnecessary. Let $pr_1 : \PP^N\times\PP^{N\vee}\rightarrow \PP^N$, $pr_2 : \PP^N\times\PP^{N\vee}\rightarrow \PP^{N\vee}$ be the projection maps. The map $pr_1$ induces a map $pr_1: Z(X)\rightarrow X$ the fibers of which are irreducible of dimension $N - n - 1$ and in fact identifies $Z(X)$ with a projective bundle over $X$. Specifically, the restriction of the line bundle $pr_2^*(O_{\PP^{N\vee}}(-1))$ to $Z(X)$ is a subbundle of $pr_1^*(\NXt^\vee)$ which induces an isomorphism $Z(X)\cong P(\NXt^\vee)$ compatible with the projection maps \cite[Appendix B, 5.5]{fulton1}. See also \cite[Section 1]{holme1}, though note the differences there in projective bundle conventions. As an immediate consequence of this identification we have $\dim(Z(X)) = N - 1$.

We define the \emph{dual variety} of $X$ in $\PP^{N\vee}$, denoted by $X^\vee$, to be the image of $Z(X)$ by $pr_2$. Thus $\dim(X^\vee)\leq N - 1$. The \emph{duality defect conjecture} claims that when $X$ is nonlinear and $n>\frac{2N}{3}$, this inequality is actually equality, or in other words the \emph{duality defect} of $X$, $\deff(X) = N-1 -\dim(X^\vee)$, is zero.

To study $Z(X)$ and $X^\vee$, we first give a more useful description of the embedding of $Z(X)$ into $\PP^N\times\PP^{N\vee}$ for our purposes. We will do this by showing that $\NXt$ is globally generated, and will furnish a surjective morphism $O_X^{N + 1}\rightarrow \NXt$. The following approach is that of Holme \cite{holme1}, \cite{holme3} and Oaland \cite{oaland1}, but we describe it here in slightly greater generality for later use. Suppose $G$ is a globally generated vector bundle on $X$ with surjection $O_X^{N+1}\rightarrow G$. Then this map induces a closed embedding $$P(G^\vee)\hookrightarrow P((O_X^{N+1})^\vee) = \Proj(\Sym_{\OOs_X}(\OOs_X^{N+1})) = X\times\PP^N.$$ After composition with the maps $$X\times\PP^N\hookrightarrow \PP^N\times\PP^N \cong \PP^N\times\PP^{N\vee},$$ this gives an embedding $P(G^\vee)\hookrightarrow \PP^N\times\PP^{N\vee},$ enabling us to consider the class of $P(G^\vee)$ in the Chow ring $A(\PP^N\times\PP^{N\vee})$. Let $pr_1,pr_2$ be the projection maps from $\PP^N\times\PP^{N\vee}$ like before. Note that $O_{P((O_X^{N+1})^\vee)}(1)$ is the pullback of $O_{\PP^{N\vee}}(1)$ by the composite map $X\times\PP^N\hookrightarrow \PP^N\times\PP^{N\vee}\rightarrow \PP^{N\vee}$.

There is an isomorphism $A(\PP^N\times\PP^{N\vee})\cong \ZZ[s,t]/(s^{N+1},t^{N+1})$ identifying $s,t$ with $pr_1^*([H])$, $pr_2^*([H^\prime])$ respectively, for hyperplanes $H\subseteq \PP^N$, $H^\prime\subseteq \PP^{N\vee}$. If $P(G^\vee)$ has dimension $d$, then $[P(G^\vee)]$ is a homogeneous element of $A(\PP^N\times\PP^{N\vee})$ of degree $r = 2N - d$ with respect to the natural grading by codimension, and so there exist $a_0,\ldots,a_r\in \ZZ$ such that
\begin{equation}[P(G^\vee)] = a_0 s^r + a_1 s^{r - 1}t + \ldots + a_{r-1} st^{r - 1} + a_r t^r.\end{equation}

These $a_j$ can be computed from the Segre classes of $G^\vee$. This follows from \emph{Scott's formula} \cite[pg. 61]{fulton1}:

\begin{theorem}
Let $E$ be a vector bundle on $X$, and let $F$ be a subbundle of $E$, with quotient bundle $G$. Suppose $G$ has rank $q$. There is a canonical closed embedding $P(F)\hookrightarrow P(E)$, which allows us to consider $[P(F)]$ as a class in $A(P(E))$. Then $$[P(F)] = \sum_{j = 0}^q c_1(O_{P(E)}(1))^j p^* (c_{q - j}(G))\in A(P(E)),$$ where $p: P(E)\rightarrow X$ is the structure morphism.
\end{theorem}

Applying this formula to the exact sequence $$0\rightarrow G^\vee\rightarrow (O_{X}^{N+1})^\vee\rightarrow F^\vee\rightarrow 0,$$ where $F$ is the kernel of the surjection $O_{X}^{N+1}\rightarrow G$, and arguing analogously to Holme \cite[Section 1]{holme1} establishes:

\begin{lemma}
We have $$[P(G^\vee)] = \sum_{j = 0}^q t^j p^* (s_{q - j}(G^\vee))\in A(X\times\PP^{N}),$$ where $q$ is the rank of $F$, and $t$ denotes the pullback of $t = pr_2^*([H^\prime])$ by the inclusion $X\times\PP^N\hookrightarrow \PP^N\times\PP^{N\vee}$. Furthermore, $a_j = \deg(s_{q - j}(G^\vee))$ for each $j = 0,\ldots, q$.
\end{lemma}

The normal bundle $\NX$ is defined from the exact sequence $$0 \rightarrow T_{X}\rightarrow i^* (T_{\PP^N})\rightarrow \NX \rightarrow 0.$$ Tensoring by a line bundle preserves exactness, so we may twist to obtain $$0 \rightarrow T_{X}(-1)\rightarrow i^*(T_{\PP^N})(-1)\rightarrow \NXt \rightarrow 0.$$ Next applying the pullback functor $i^*(-)$ to the Euler sequence for $\PP^N$ \cite[II, Example 8.20.1]{hartshorne1} and then twisting yields the exact sequence $$0\rightarrow O_X(-1) \rightarrow O_X^{N+1}\rightarrow i^* (T_{\PP^N}) (-1)\rightarrow 0.$$ By identifying the term $i^*(T_{\PP^N})(-1)$ in both of these sequences we obtain the desired surjection $O_X^{N+1}\rightarrow \NXt$.

Replacing $G$ with $\NXt$ in equation (2) yields the simpler expression $$[Z(X)] = [P(\NXt^\vee)] = a_1 s^N t + \ldots + a_N s t^N.$$ When referring to the class of $Z(X)$, we adopt the notation $a_{j+1} = \delta_j(X)$, and refer to the $\delta_j(X)$ as the \emph{delta invariants} or the \emph{degrees of the polar classes} of $X$, see \cite[Section 3]{holme3}. It follows from the projection formula that in fact $\delta_{n+1}(X) s^{N-n-1}t^{n+2} + \ldots + \delta_{N-1}(X)s t^N = 0$, and so the expression for the class of the conormal variety simplifies further to $$[Z(X)] = \delta_0(X)s^N t + \ldots + \delta_n(X) s^{N-n}t^{n + 1}.$$

From Holme \cite[Theorem 1.1]{holme1}, we are able to read off the duality defect of $X$ from the $\delta_j(X)$:

\begin{theorem}
Let $r\in \{0,\ldots,n\}$. If $\delta_0(X) = \ldots = \delta_{r - 1}(X) = 0$ and $\delta_{r} (X) \neq 0$, then $X^{\vee}$ has dimension $N - 1 - r$. Thus in this case, $\deff(X) = r$.
\end{theorem}

\begin{remark}
The argument used to prove this establishes that the $\delta_j(X)$ represent intersection numbers of projective subvarieties with linear subvarieties in general position, and thus that $\delta_j(X)\geq 0$ for every $j$. This same argument will work for any globally generated vector bundle $G$ on $X$, and shows that in the expression (2) for the class $[P(G^\vee)]$, each $a_j\geq 0$. By Lemma 3.2, this means that $\deg(s_{q-j}(G^\vee))\geq 0$ for $j = 0,\ldots,q$, where $q$ is as in Lemma 3.2.
\end{remark}

Applying Lemma 3.2 to the surjection $O^{N+1}_X\rightarrow \NXt$ gives us $\delta_j(X) = \deg(s_{n-j}(\NXt))$ for each $j$. From the exact sequence $$0\rightarrow F\rightarrow O_X^{N+1}\rightarrow \NXt\rightarrow 0,$$ where $F$ denotes the kernel of the surjection $O_X^{N+1}\rightarrow \NXt$, we obtain a surjection $$O_X^{N+1} \cong (O_X^{N+1})^{\vee}\rightarrow F^{\vee}.$$ This shows $F^{\vee}$ is globally generated, and therefore by Lemma 3.2, the $\deg(s_j(F))$ are the coefficients of the class of $P(F)$ in $A(\PP^N\times\PP^{N\vee})$. By the sum formula for Chern classes, $$c_t(F)c_t(\NXt) = c_t(O_X^{N+1}) = 1,$$ which implies $c_t(\NXt) = s_t(F)$. Thus in particular, $\deg(c_j(\NXt)) = \deg(s_j(F))$ for $j = 0,\ldots,N-n$. This in addition to Remark 3.4 gives:

\begin{lemma}
For every $j = 0,\ldots,n$, $\delta_j(X) = \deg(s_{n - j}(\NXt^\vee))\geq 0$, and for $j = 0,\ldots,N - n$, we have $\deg(c_j(\NXt))\geq 0$.
\end{lemma}

From Holme \cite[Theorem 5.1]{holme3} in fact:

\begin{theorem}
For every $j = \deff(X),\ldots,n$, $\delta_j(X) = \deg(s_{n - j}(\NXt^\vee))> 0$.
\end{theorem}

Lastly, for our purposes Theorem 2.3 will allow us to assume $c(\NXt) = 1 + c_1 h + \ldots + c_{N-n} h^{N-n}$ for some $c_j\in \ZZ$ and $h = i^*([H])$, provided the codimension of $X$ is sufficiently small. We use the following Lemma to extract precisely what we need from Theorem 2.3 and translate it to the algebraic context for convenience.

\begin{lemma}
Let $i:X\hookrightarrow \PP^N$ be a smooth subvariety of dimension $n \geq \frac{3N - 2}{4}$. Then for each $j = 1,\ldots,N-n$, there exists $c_j\in \ZZ$ such that for any $j_1,\ldots,j_{N-n}\in\ZZ_{\geq 0}$ and any homogeneous $\alpha\in A(X)$, $$\deg(\alpha\prod_{q = 1}^{N-n}c_q(\NX)^{j_q}) = \deg(\alpha\prod_{q = 1}^{N-n}(c_q h^{q})^{j_q}).$$ That is, for the purpose of computing the degrees of such monomials, we may assume each $c_j(\NX) = c_j h^j$. The analogous result holds for the bundle $\NXt$.
\end{lemma}
\begin{proof}
By the self-intersection formula, in fact $c_{N-n}(\NX) = i^*(i_*([X])) = \deg(X) h^{N-n}$. So we need only address the cases of $c_{j}(\NX)$ for $j = 1,\ldots,N-n-1$. For such a $j$ we have the commutative diagram

\begin{center}
\begin{tikzcd}
A^j(\PP^N) \arrow[r, "i^*"] \arrow[d, rightarrow, "{cl}^*"] & A^{j}(X) \arrow[d, "{cl}^*"] \\
H^{2j}(\PP^N,\ZZ) \arrow[r, "i^*"] & H^{2j}(X,\ZZ)
\end{tikzcd}
\end{center}

By Theorem 2.3 the bottom map of this diagram is an isomorphism, and so too is ${cl}^*:A^j(\PP^N)\rightarrow H^{2j}(\PP^N,\ZZ)$. Thus there exists a $c_j\in\ZZ$ such that $c_j(\NX) - c_j h^j\in\ker({cl}^*)$. The desired result then follows from the observation that for homogeneous $\alpha,\beta,\gamma\in A(X)$, with $\alpha,\beta$ in the same homogeneous component, and for $q \in \ZZ_{\geq 1}$, if $\alpha - \beta \in \ker({cl}^*) = \ker({{cl}_*})$ then $\deg(\gamma \alpha^q) = \deg(\gamma \beta^q)$. That the analogous result holds for $\NXt$ is a consequence of the formula for the Chern classes of a bundle after taking a tensor product with a line bundle \cite[Example 3.2.2]{fulton1}.
\end{proof}

\section{A generalized algorithm and bounds for the degrees of counterexamples}

Here we will show how the facts presented in the preceding section connect the duality defect conjecture to the theory of linear recurrence sequences and will give a generalization of the algorithm of Oaland \cite{oaland1} to the higher codimension cases of the conjecture. We then prove Theorem 4.8, stated in the introduction, which bounds the degrees of possible counterexamples to the conjecture that satisfy the constraints of Lemma 3.7.

Suppose $i: X\hookrightarrow\PP^N$ is a smooth subvariety of dimension $n\geq \frac{3N-2}{4}$. By Lemma 3.7 for our purposes we may assume there exist $c_j\in\ZZ$ such that $c(\NXt) = 1 + c_1 h + \ldots + c_{N - n} h^{N - n}$, where $h = i^*([H])$ for any hyperplane $H\subseteq\PP^N$. Applying the formula for the Chern class of the dual of a bundle \cite[Remark 3.2.3]{fulton1} gives $c(\NXt^\vee) = 1 - c_1 h + \ldots + (-1)^{N - n} c_{N - n} h^{N - n}$. So now, by the formula of Lemma 2.1, we have $$s_j(\NXt^\vee) = \sum_{q = 1}^{j}(-1)^{q+1} c_{q} h^{q}s_{j-q}(\NXt^\vee)$$ for every $j \in \ZZ_{>0}$. Thus there exist $s_j\in \ZZ$ such that $s(\NXt^\vee) = 1 + s_1 h + \ldots + s_n h^n$. Furthermore, since the Chern classes of $\NXt^\vee$ vanish beyond its rank $N-n$, the $s_j$ form the recurrence sequence (1):

\begin{equation*}
\begin{aligned}
s_0 &= 1,\\
s_j &= \sum_{q = 1}^{j} (-1)^{q+1} c_{q}s_{j-q}, \text{ for } j = 1,\ldots,N- n - 1\\
s_j &= \sum_{q = 1}^{N - n} (-1)^{q + 1}c_{q}s_{j - q}, \text{ for } j \geq N - n.
\end{aligned}
\end{equation*}

By the projection formula, for each $j$, $\delta_j(X) = \deg(s_{n - j}(\NXt^\vee)) = \deg(X)s_{n - j}$. So in particular, the positivity of the delta invariants is the same as that of the $s_{n-j}$. From Lemma 3.5, this implies that the $s_{n-j}$ are nonnegative, and additionally by Theorem 3.3, that for $j = 0,\ldots,\deff(X) - 1$, $s_{n-j}=0$, and for $j = \deff(X),\ldots,n$, that $s_{n-j}>0$ by Theorem 3.6. Lemma 3.5 also shows that each $c_j \geq 0$.

This means every smooth subvariety $i: X\hookrightarrow\PP^N$ of dimension $n\geq \frac{3N-2}{4}$ gives rise to $c_j\in\ZZ_{\geq 0}$ such that the corresponding linear recurrence sequence (1) satisfies the above positivity constraints. If we can show that no such linear recurrence sequences exist for positive values of $\deff(X)$, then there can be no positive defect subvarieties.

Ein \cite[Theorem 2.4]{ein1} has proven the following parity result concerning duality defect:

\begin{theorem}
If $X$ is a smooth nonlinear subvariety of $\PP^N$ with positive duality defect $\deff(X)$, then $\deff(X) \equiv \dim(X) \mod 2$.
\end{theorem}

For a general point $x\in X$ and a general tangent hyperplane $H$ in $\PP^N$ of $X$ at $x$, we define the \emph{contact locus} of $H$ with $X$ to be the singular locus of the intersection of $X$ and $H$, $(X\cap H)_{\Sing}$. The above result is an immediate corollary to Ein's more general result \cite[Theorem 2.3]{ein1}:

\begin{theorem}
Suppose $X$ is a smooth nonlinear subvariety of $\PP^N$ of dimension $n$ with positive duality defect. Then if $x$ and $H$ are as above, the contact locus $L = (X\cap H)_{\Sing}$ is a linear subvariety of dimension $\deff(X)$. Let $T$ be a line in $L$, which we identify with $\PP^1$. Then
$$N_{L/X}\big|_T\cong O_{\PP^1}^{\frac{n - \deff(X)}{2}}\oplus O_{\PP^1}^{\frac{n - \deff(X)}{2}}(1).$$
\end{theorem}

Since $L$ has dimension $\deff(X)$, we also have $N_{L/\PP^N}\big|_T \cong O_{\PP^1}^{N - \deff(X)}(1).$ These normal bundles fit into the following exact sequence, see \cite[Appendix B, 7.4]{fulton1}: $$0\rightarrow N_{L/X}\big|_T\rightarrow N_{L/\PP^N}\big|_T\rightarrow \NX\big|_T\rightarrow 0.$$

It then follows from the sum formula for Chern classes that
\begin{equation}
\deg(c_{1}(\NX\big|_T)) = \frac{2N - n - \deff(X)}{2}.
\end{equation}

We also rely on a result of Zak \cite[Corollary 7.4]{fulton2}:

\begin{theorem}
If $X$ is a smooth nonlinear subvariety of $\PP^N$, then $\dim(X^\vee)\geq \dim(X)$.
\end{theorem}

Suppose again that $i:X\hookrightarrow \PP^N$ is a smooth subvariety of dimension $n\geq \frac{3N-2}{4}$, and let $m = N-n$. By Theorem 4.3, $\dim(X^\vee)\geq n$, therefore $0 \leq \deff(X) \leq m - 1$. Suppose further that $\deff(X) > 0$, then by Theorem 4.1 $\deff(X) \equiv n \mod 2$, and from (3) and \cite[Example 3.2.2]{fulton1} we have $$c_1 = \frac{N - m - \deff(X)}{2}.$$

Now assume there exist integers $B_{c_1,j}$ depending only on $c_1$ for $j = 2,\ldots, m$ such that each $c_j\leq B_{c_1,j}$. We then obtain a brute-force algorithm which can be used to prove the duality defect conjecture in the case of codimension $m$ subvarieties of $\PP^N$. This is a generalized version of the algorithm given by Oaland \cite[Kapittel 5]{oaland1} for the codimension $3$ case of the conjecture.\\

\noindent\textbf{Algorithm 1.}
\noindent
Input: integers $m\in\ZZ_{\geq 3}$, $N\in\ZZ_{\geq 10}$ such that $N-m \geq \frac{3N - 2}{4}$\\
\noindent
Output: \emph{True} or \emph{False}
\begin{itemize}
    \item for each $r = 1,\ldots,m-1$ with $r \equiv N - m \mod 2$:
    \begin{itemize}
        \item set $c_1 = \frac{N - m - r}{2}$
        \item for each tuple $(c_2,\ldots,c_m)\in\ZZ_{\geq 0}^{m-1}$ such that each $c_j\leq B_{c_1,j}$:
            \begin{itemize}
                \item if $s_j(c_1,\ldots,c_m) > 0$ for $j = 0,\ldots,n-r$ and $s_j(c_1,\ldots,c_m) = 0$ for $j = n-r + 1,\ldots,n$:
                \begin{itemize}
                    \item return \emph{False}
                \end{itemize}
            \end{itemize}
    \end{itemize}
    \item return \emph{True}
\end{itemize}

Here the algorithm returns \emph{True} when there are no problematic \emph{Chern numbers} $c_j$ satisfying the requisite positivity conditions for a positive defect subvariety of codimension $m$ in $\PP^N$, and thus the conjecture is true for that case. However, if problematic Chern numbers are found, we have not made any guarantee that they must come from a positive defect subvariety, and so the truth of the conjecture in that case falls outside the scope of this algorithm and the test is inconclusive.

To apply this algorithm, we need to find bounds $B_{c_1,j}$. Oaland \cite{oaland1} provides the bound $B_{c_1,2} = c_1^2$ by using the \emph{numerical nonnegativity} of the Schur polynomials in the Chern classes of $\NXt$, though this may also be obtained from the fact that $s_2 = c_1^2 - c_2$. The idea to use Schur polynomials to obtain further bounds appears to be implicit in Oaland's work \cite[4.3 Kodimensjon st\o rre enn 3]{oaland1}; however, no other bounds are provided and instead the possibilities for $c_3$ given a choice of $c_1,c_2$ in the codimension $3$ algorithm presented there are found as the nonnegative integer roots of $s_n(c_1,c_2,c_3)\in\ZZ[c_3]$. By continuing with this idea to use the nonnegativity of the Schur polynomials we are able to find bounds $B_{c_1,j}$ for every $j\geq 2$.

\begin{lemma}
Suppose $i:X\hookrightarrow \PP^N$ is a smooth subvariety of dimension $n\geq \frac{3N - 2}{4}$ with $N-n \geq 2$. Suppose the $c_j$ are as before. Then we have $c_j \leq c_1^{j}$ for every $j\geq 2$. That is, we may take $B_{c_1,j} = c_1^{j}$ for every $j\geq 2$.
\end{lemma}
\begin{proof}
Let $2\leq j\leq N-n$. If $\lambda = (j-1,1,0,\ldots,0)$ represents the partition $\lambda_1 = j-1 \geq \lambda_2 = 1 \geq \lambda_3 = 0 \geq \ldots \geq \lambda_j = 0$ of $j$, then by Theorem 2.2 and Lemma 3.7 we have $$\Delta_\lambda(\NXt) = c_1(\NXt)c_{j-1}(\NXt) - c_{j}(\NXt) = (c_1c_{j-1} - c_j)h^j \in A^{\geq}_{n - j}(X).$$ Therefore $\deg(\Delta_\lambda(\NXt))\geq 0$ and so $c_1c_{j-1} - c_j\geq 0$. Thus by induction on $j$, $c_j \leq c_1^j$ for every $j = 2,\ldots,N-n$.
\end{proof}

Interestingly, these inequalities also follow from a recent result of Huh \cite[Theorem 21]{huh1} which imposes restrictions on the coefficients of the class of a subvariety of a product of projective spaces. One case of Huh's result is the following:

\begin{theorem}
If $Y$ is a subvariety of $\PP^N\times\PP^N$ and we write $[Y] = \sum_{j}a_j [\PP^{\dim(Y) - j}\times\PP^{j}]\in A(\PP^N\times\PP^N)$ for some $a_j\in \ZZ$, then the $a_j$ are nonnegative and form a log-concave sequence with no internal zeros. Here \emph{log-concavity} refers to the property that $a_j^2 \geq a_{j - 1}a_{j+1}$ for each $j$.
\end{theorem}

If $F$ denotes the kernel of the surjection $O_X^{N+1}\rightarrow \NXt$, the Chern classes of $\NXt$ are the Segre classes of $F$, and since $F^\vee$ is globally generated, together with Lemma 3.2 this shows that the $$c_j\deg(X) = \deg(c_j(\NXt)) = \deg(s_j(F))$$ form a log-concave sequence with no internal zeros, and thus so too do the $c_j$. In particular, $c_j^2 \geq c_{j-1}c_{j+1}$ for every $j = 1,\ldots,N-n-1$. The inequalities $c_j\leq c_1^j$ then follow by an induction argument.

\begin{remark}
If $X$ is a smooth subvariety of $\PP^N$ for arbitrary $N$, then by Holme \cite[Theorem 4.2]{holme3} we have that $\delta_j(X) = \delta_{j+1}(X^\prime)$ for each $j$, where $X^\prime$ is any general hyperplane section of $X$, treated as a subvariety of $\PP^{N-1}$. So if $X$ has duality defect greater than $1$, $X^\prime$ will also be a positive defect subvariety. By Theorem 4.1, if $X$ has positive duality defect, and if $\dim(X)$ is even, then so too is $\deff(X)$, and thus $\deff(X)\geq 2$. This means that to apply Algorithm 1 to a particular codimension, it suffices to check only half the possible $N$. For example, in the codimension $3$ case one only needs to apply the algorithm successfully to an even $N$ in order to prove the conjecture for $N+1$ also.
\end{remark}

Oaland \cite[6.2 Udata for kodimensjon 3]{oaland1} has run Algorithm 1 for codimension $3$ subvarieties in $\PP^N$ for $N = 10,12,\ldots,140$, which by Remark 4.6 proves the conjecture for $N = 10,11,\ldots,141$. We have reproduced Oaland's results by running Algorithm 1 using the \emph{SageMath} computer algebra system \cite{sage1}, and have computed slightly further from $N = 142,144,\ldots,200$. We have also run the algorithm in the codimension $4$ and $5$ cases for $N = 14,15,17,19,\ldots,47,49$ and $N = 18,20,22$ respectively. The cases $N = 13$, and $N = 16,17$ are still of interest to the duality defect conjecture in these codimensions, but our current methods restrict us to $N$ satisfying the constraint of Lemma 3.7. Altogether we have the following.

\begin{theorem}
The duality defect conjecture is true for $\PP^N$ in:
\begin{enumerate}
\item the codimension $3$ case when $N = 10,11,\ldots,201$,
\item the codimension $4$ case when $N = 14,15,\ldots,50,$
\item and in the codimension $5$ case when $N = 18,19,\ldots,23$.
\end{enumerate}
\end{theorem}

Using the inequalities for the $c_j$, we are able to obtain bounds for the degrees of positive defect subvarieties satisfying the constraints of Lemma 3.7 in terms of the duality defect and ambient space dimension:

\begin{theorem}
If $X$ is a smooth nonlinear codimension $m$ subvariety of $\PP^N$ with duality defect $r > 0$ and $N - m \geq \frac{3N-2}{4}$, then $$\deg(X)\leq \sum_{j = 0}^{m}\left(\frac{N - m - r}{2}\right)^{j}.$$
\end{theorem}
\begin{proof}
Suppose we are given such an $X$, and let the $c_j$ be as before. By the self-intersection formula, $c_m(\NX) = \deg(X)h^{m}$. From another application of \cite[Example 3.2.2]{fulton1}, as $\NX \cong \NXt\otimes O_{X}(1)$, we obtain $c_m(\NX) = (1 + c_1 + \ldots + c_m)h^m$. Therefore $\deg(X) = 1 + c_1 + \ldots + c_m$. Since $X$ has duality defect $r>0$, $c_1 = \frac{N - m - r}{2}$, and so by applying the bounds $B_{c_1,j}$ we obtain the desired inequality.
\end{proof}

Results giving sufficient conditions for a smooth projective variety to be a complete intersection in terms of bounds on the degree of the variety have been established previously, such as that of Bertram et al. \cite[Corollary 3]{bertram1} and Holme and Schneider \cite[Theorem 5.1]{holme2}. In particular, the result of Bertram et al. shows that any smooth subvariety $X$ of $\PP^N$ of codimension $m$ with $\deg(X)\leq \frac{N}{2m}$ is a complete intersection. A stronger result of this form, in combination with Theorem 4.8, could potentially give a proof of the duality defect conjecture within the constraints of Lemma 3.7.

\section{Duality defect in codimension $3$}

Using an observation about the positivity of a particular family of homogeneous order three integer linear recurrence sequences, we prove the duality defect conjecture for the codimension 3 case when the dimension of the projective ambient space $\PP^N$ is odd. This restriction to the case of odd ambient space dimension ensures that the possibilities for counterexamples that we must rule out are subvarieties with duality defect $2$, a fact which makes the corresponding number-theoretic question more tractable.

\begin{theorem}
The duality defect conjecture in the codimension $3$ case is true for $\PP^N$ when $N$ is odd.
\end{theorem}

Let $i: X\hookrightarrow \PP^N$ be a smooth nonlinear positive defect subvariety of codimension $3$, where $N\geq 10$ is odd. By Theorems 4.1, 4.3, we see that $\deff(X) = 2$. Let $n = N - 3$. By Larsen's theorem, see Lemma 3.7, we can assume $c(\NXt) = 1 + c_1 h + c_2 h^2 + c_3 h^3$ and $s(\NXt^\vee) = 1 + s_1 h + s_2 h^2 + \ldots + s_n h^n$ where $h = i^*([H])$ for a hyperplane $H\subseteq \PP^N$, for some $s_j, c_q\in\mathbb{Z}_{\geq 0}$. Here the $s_j$ form part of the recurrence sequence $s_0 = 1, s_1 = c_1, s_2 = c_1^2 - c_2$, $$s_j = c_1 s_{j - 1} - c_2 s_{j - 2} + c_3 s_{j - 3}$$ for $j > 2$. Furthermore, $\delta_j(X) = \deg(X)s_{n - j}$ for each $j$.

Since $\deff(X) = 2$, $\delta_0(X) = \delta_1(X) = 0$, and $\delta_j (X) > 0$ for $j > 1$. Thus we see $s_n, s_{n - 1} = 0$, and $s_j > 0$ for $j = 0,\ldots,n - 2$. From these properties it follows that $c_1,c_2,c_3 > 0$ as well.

To prove our theorem, we investigate the positivity of this sequence. Note we get the same later terms if we start with initial conditions $u_{0} = 0, u_{1} = 0, u_2 = 1$ instead. That is, it suffices to study the sequence $$u_j = c_1 u_{j - 1} - c_2 u_{j - 2} + c_3 u_{j - 3}$$ with $u_{0} = 0, u_{1} = 0, u_2 = 1$, since we have $u_{j+2} = s_j$ for each $j$. The \emph{characteristic polynomial} of this sequence is $\rho = t^3 - c_1 t^2 + c_2 t - c_3$.

From the results of Section 4, we know the codimension $3$ duality defect conjecture is true for $N = 10,11,\ldots,201$. Therefore, if we could conclude via a purely number-theoretic result that the only possibilities for a recurrence sequence with the properties of $(u_j)_{j\in\ZZ_{\geq 0}}$ are those with $n\leq 198$, we would have proved our theorem. In fact, a much more precise result of this type is true.

\begin{lemma}
Let $(u_j)_{j\in\mathbb{Z}_{\geq0}}$ be the homogeneous integer linear recurrence sequence defined by the initial conditions $u_0 = 0, u_1 = 0, u_2 = 1$, and $u_j = c_1 u_{j - 1} - c_2 u_{j - 2} + c_3 u_{j - 3}$ for $j>2$, where $c_1,c_2,c_3\in\mathbb{Z}_{>0}$. Suppose there exists an $m\in\ZZ_{>2}$ with $u_m = u_{m+1} = 0$ and $u_{j} > 0$ for $j = 2,\ldots,m-1$. Then $m = 4$ or $6$.
\end{lemma}

\begin{proof}
First note that $m=3$ is not possible, as $u_3 = c_1$ which is assumed to be positive. Let $$A =
\begin{bmatrix}
c_1 & -c_2 & c_3\\
1 & 0 & 0\\
0 & 1 & 0
\end{bmatrix}.$$ Then for every $j > 2$, we have $$A
\begin{bmatrix}
u_{j-1}\\
u_{j-2}\\
u_{j-3}
\end{bmatrix} =
\begin{bmatrix}
u_j\\
u_{j-1}\\
u_{j-2}
\end{bmatrix}.$$ So in particular, $$A^m
\begin{bmatrix}
1\\
0\\
0
\end{bmatrix} = d
\begin{bmatrix}
1\\
0\\
0
\end{bmatrix},$$ where $d = c_3 u_{m-1}\in\ZZ_{>0}$. This reflects the fact that the existence of $m$ forces the sequence $(u_j)_{j\in\ZZ_{\geq 0}}$ to repeat up to multiplication by $d$ every $m$ steps. Note $$
\begin{bmatrix}
1\\
0\\
0
\end{bmatrix},
\begin{bmatrix}
u_3\\
1\\
0
\end{bmatrix},
\begin{bmatrix}
u_4\\
u_3\\
1
\end{bmatrix}$$ are linearly independent vectors, and as they are all eigenvectors associated to the eigenvalue $d$ of $A^m$, if we let $$B = \frac{1}{d^{\frac{1}{m}}}A,$$ then $B^m = \id_3$, the identity element of the multiplicative group of $3\times 3$ invertible real matrices, $\GL(3,\RR)$. By the minimality of $m$, $B$ has order $m$ as an element of $\GL(3,\RR)$. The eigenvalues of $A^m$ are exactly the $\alpha_1^m, \alpha_2^m, \alpha_3^m$, where $$(t - \alpha_1)(t - \alpha_2)(t - \alpha_3) = \det(t\id_3 - A) = t^3 - c_1 t^2 + c_2 t - c_3 = \rho$$ is the factorization over $\CC$ of the characteristic polynomial of $A$, which is also the characteristic polynomial of the sequence $(u_j)_{j\in\ZZ_{\geq 0}}$. Here the $\alpha_j$ must be distinct otherwise the positivity of either the $u_j$ or the $c_j$ is violated. Therefore, as $d$ is the only eigenvalue of $A^m$, $\rho$ must divide $t^m - d$.

If $d^{\frac{1}{m}}\not\in \ZZ$, then $\rho$ and $t^3 - d^{\frac{3}{m}}$ are irreducible elements of $\ZZ[t]$. Here $d^{\frac{3}{m}} = c_3 \in\ZZ_{>0}$. Then since $d^{\frac{1}{m}}$ is a root of $\rho$, we see $\rho$ must be equal to $t^3 - d^{\frac{3}{m}}$, which is impossible since we assume $c_1,c_2 > 0$.

This implies $d^{\frac{1}{m}}\in\ZZ$, and therefore $B$ is an element of order $m$ in $\GL(3, \QQ)$. Such matrices are well understood via elementary methods, and in particular the possible finite orders of elements of $\GL(3,\QQ)$ are $1,2,3,4,6$ \cite{koo1}. Thus the only possibilities for $m$ are $4,6$.
\end{proof}

This result is sharp in that both possibilities for $m$ can occur, for instance:

\begin{example}
\begin{itemize}
\item[]
\item[] $m = 4$: take $c_1 = 3, c_2 = 9$, and $c_3 = 27$. Then $$(u_j)_{j\in\ZZ_{\geq 0}} = (0, 0, 1, 3, 0, 0, 81, 243, 0, 0, 3561, 19683, 0, 0, \ldots),$$
\item[] $m = 6$: take $c_1 = 4, c_2 = 8$, and $c_3 = 8$. Then $$(u_j)_{j\in\ZZ_{\geq 0}} = (0, 0, 1, 4, 8, 8, 0, 0, 64, 256, 512, 512, 0, 0, \ldots).$$
\end{itemize}
\end{example}

In view of Theorem 4.1, we have restricted ourselves to the case when the ambient projective space has odd dimension to simplify the corresponding number-theoretic problem. As demonstrated by Lemma 5.2 it is a relatively simple matter to understand the possibilities for homogeneous order three integer linear recurrence sequences that have two consecutive zeros in addition to their initial conditions. However, by Remark 4.6, the entire codimension $3$ duality defect conjecture would follow if we could prove it for the case of even ambient space dimensions instead. The number-theoretic question in this case appears to be far more complicated, closer to the general problem of finding the zeros of linear recurrence sequences, a problem which has received much study but remains difficult \cite[Chapter 2]{everest1}.

Nevertheless, there are results in this direction, such as one due to Mignotte et al. \cite[Theorem 4]{mignotte1} which gives a computable bound for the indices of zeros of order three recurrence sequences subject to a certain nondegeneracy condition. It would be interesting to know if such a result in conjunction with the positivity constraints presented here for recurrence sequences representing positive defect subvarieties is enough to prove more cases of the duality defect conjecture.

\end{document}